\documentclass[oneside,english]{amsart}
\usepackage[T1]{fontenc}
\usepackage[latin9]{inputenc}
\usepackage{geometry}
\usepackage{amsthm}
\usepackage{amstext}
\usepackage{amssymb}
\usepackage{color}

\makeatletter
\numberwithin{equation}{section}
\numberwithin{figure}{section}
\theoremstyle{plain}
\newtheorem{thm}{\protect\theoremname}
  \theoremstyle{plain}
  \newtheorem{prop}[thm]{\protect\propositionname}
  \theoremstyle{remark}
  \newtheorem{rem}[thm]{\protect\remarkname}
  \theoremstyle{plain}
  \newtheorem{cor}[thm]{\protect\corollaryname}
  \theoremstyle{definition}
  \newtheorem{example}[thm]{\protect\examplename}
\theoremstyle{definition}
  \newtheorem{notation}[thm]{Notation}
\theoremstyle{plain}
  \newtheorem{lemma}[thm]{Lemma}

\theoremstyle{plain}
  \newtheorem*{thm*}{Theorem}

\makeatother

\usepackage{babel}
  \providecommand{\corollaryname}{Corollary}
  \providecommand{\examplename}{Example}
  \providecommand{\propositionname}{Proposition}
  \providecommand{\remarkname}{Remark}
\providecommand{\theoremname}{Theorem}

\newtheorem*{LCP}{Laurent Cancellation Problem}

\newcommand{\Ker}{\textrm{Ker}}

\newcommand{\C}{\mathbb{C}}
\newcommand{\Z}{\mathbb{Z}}
\newcommand{\kk}{\mathbf{k}}

\begin{document}

\author{Adrien Dubouloz and Pierre-Marie Poloni} 
\address{Adrien Dubouloz \\ IMB UMR5584, CNRS, Univ. Bourgogne Franche-Comt\'e, F-21000 Dijon,France.}
\email{adrien.dubouloz@u-bourgogne.fr} 
\address{ Pierre-Marie Poloni\\ Universit\"at Bern \\ Mathematisches Institut \\ Sidlerstrasse 5 \\ CH-3012 Bern \\ Switzerland\\} 
\email{pierre.poloni@math.unibe.ch}
\thanks{This research was supported in part by the ANR Grant BirPol ANR-11-JS01-004-01. The second author was supported by the Universit\"at Basel Forschungsfonds.}

\title{Affine-ruled varieties without the Laurent cancellation property}
\begin{abstract}
We describe a method to construct hypersurfaces of the complex affine $n$-space with isomorphic $\C^*$-cylinders. Among these hypersurfaces, we find new explicit counterexamples to the Laurent Cancellation Problem, i.e.~hypersurfaces that are nonisomorphic, although their $\C^*$-cylinders are isomorphic as abstract algebraic varieties. We also provide examples of nonisomorphic varieties $X$ and $Y$ with isomorphic cartesian squares $X\times X$ and $Y\times Y$.
\end{abstract}
\maketitle

\section*{Introduction}

In this paper, we consider the following cancellation problem:
\begin{LCP} Suppose that the $\C^*$-cylinders $X\times\C^{\ast}$ and $Y\times\C^*$ over two complex affine varieties $X$ and $Y$ are isomorphic. Does it follow that $X$ and $Y$ are isomorphic as abstract algebraic varieties?\end{LCP}

This question can of course be equivalently reformulated as the question of the uniqueness of the coefficient ring in a Laurent polynomial ring. Indeed, if we let $X=\text{Spec}(A)$ and $Y=\text{Spec}(B)$ for some finitely generated complex algebras, then the  Laurent Cancellation Problem simply asks whether having isomorphic Laurent polynomial rings $A[t,t^{-1}]\simeq B[t,t^{-1}]$ implies that $A$ and $B$ are isomorphic themselves. 

The answer  is known to be positive in many cases. First of all, it is easy to see that tori have the Laurent Cancellation property, i.e.~that  the Laurent Cancellation Problem has a positive answer, when $X$ is isomorphic to an algebraic torus $(\C^*)^d$ (see e.g.~Lemma 4.5 in \cite{BD}).  Then,  Gene Freudenburg \cite{Freudenburg} has proved that affine curves (i.e.~complex affine algebraic varieties $X$ of dimension 1) have also the Laurent Cancellation property. Moreover, using ideas similar to that for Iitaka and Fujita's strong cancellation theorem  \cite{IF}, the first author provided in \cite{Dubouloz-Tori} an affirmative answer for large classes of varieties. In particular, Laurent Cancellation does hold if $X$ is of log-general type \cite[Proposition 2]{Dubouloz-Tori} or if $X$ is a smooth factorial affine surface with logarithmic Kodaira dimension different from $1$ (see \cite[Proposition 12]{Dubouloz-Tori}).  On the other hand,  counterexamples were given in \cite{Dubouloz-Tori} in the form of pairs of nonisomorphic smooth factorial affine varieties $X$ and $Y$ of dimension $d\geq 2$ and logarithmic Kodaira dimension $d-1$ such that $X\times\C^*$ and $Y\times\C^*$ are isomorphic \cite[Propositions 6 and 9]{Dubouloz-Tori}.   

The main purpose of the present paper is to provide a general method to construct explicit counterexamples to the Laurent Cancellation Problem. All these  examples will be realized as hypersurfaces of affine space  $\mathbb{A}^{n+1} $ that are defined by an equation of the form $t^{\ell}f(x_1,\ldots,x_n)=1$, where $f$ is a regular function which is semi-invariant for some action of the algebraic multiplicative group $\mathbb{G}_{m}$ on $\mathbb{A}^n=\textrm{Spec}(\C[x_1,\ldots,x_n])$. We will show that some of the examples in \cite{Dubouloz-Tori} can actually be reinterpreted as being particular cases of our construction, and will also obtain new examples of varieties that fail the Laurent Cancellation property, notably   affine algebraic varieties of negative logarithmic Kodaira dimension. 

As a byproduct of our construction, we will also get explicit examples illustrating the following result. 

\begin{thm*} There exist smooth affine algebraic varieties (of every dimension $d\geq2$) $X$ and $Y$ which are  nonisomorphic, although their cartesian product with themselves, $X\times X$ and $Y\times Y$, are isomorphic. 
\end{thm*}

\smallskip

The paper is organized as follows:

In the first section, we study hypersurfaces $\widetilde{X}_{f,\ell}$ of $\C^{n+1}$ that are defined by the equation $t^{\ell}f=1$, where $f\in\C[x_1,\ldots,x_n]$ is a polynomial which is semi-invariant for an effective algebraic action of $\C^*$ on $\C^n$. We establish sufficient conditions under which the $\mathbb{C}^{*}$-cylinders $\widetilde{X}_{f,\ell}\times\mathbb{C}^{*}$ and $\widetilde{X}_{f,\ell'}\times\mathbb{C}^{*}$ are isomorphic. We also develop a general strategy to prove that two given $\widetilde{X}_{f,\ell}$ and $\widetilde{X}_{f,\ell'}$ are not isomorphic.

The second section is devoted to the case of surfaces $\widetilde{X}_{f,\ell}$ where $f\in\mathbb{C}[x,y]$ is of the form $f=x^p+y^q$. Our main result is the following.
\begin{thm*} Let $p\geq2$ be a prime number, let $q\geq3$ be an
integer relatively prime with $p$ and let $\ell\geq2$ be relatively prime with $m=pq$. Then, the smooth factorial affine surfaces of respective equation $t(x^p+y^q)=1$ and $t^{\ell}(x^p+y^q)=1$ are  not isomorphic, although they have 
isomorphic $\C^{*}$-cylinders. 
\end{thm*}

In Section 3, we focus on the case of varieties of negative logarithmic Kodaira dimension. We first improve a result of \cite{Dubouloz-Tori} by showing that Laurent Cancellation does hold for all smooth affine surfaces of negative Kodaira dimension, and then construct explicit higher dimensional examples that fail Laurent Cancellation.

Finally, examples of nonisomorphic affine varieties with isomorphic squares are given in Section 4.

\smallskip
\noindent {\bf Acknowledgments.} We are grateful to J\'er\'emy Blanc and Jean-Philippe Furter for their helpful suggestions about Proposition \ref{prop:carre-courbes}.



\section{Semi-invariants of $\mathbb{G}_m$-actions and $\mathbb{A}^1_*$-cylinders}

Unless otherwise specified, we will work over the field $\mathbb{C}$ of complex numbers. We will denote by $\mathbb{A}^n=\mathbb{A}_{\C}^n=\textrm{Spec}(\C[x_1,\ldots,x_n])$ the affine $n$-space and by $\mathbb{A}^1_{*}=\mathbb{A}^1\setminus\{0\}=\textrm{Spec}(\C[t,t^{-1}])$ the affine line minus the origin. 

\subsection{Isotrivial fibrations associated to semi-invariant regular functions}

Let $X$ be a complex affine variety endowed with an effective action $\mu:\mathbb{G}_{m}\times X\rightarrow X$ of the multiplicative group $\mathbb{G}_{m}=\mathbb{G}_{m,\mathbb{C}}=\mathrm{Spec}(\mathbb{C}[t^{\pm1}])$. Its coordinate ring $A=\mathcal{O}(X)$ is then equipped with a natural $\mathbb{Z}$-grading by the subspaces $A_{m}=\left\{ f\in A,\, f(\mu(t,x))=t^{m}f(x)\,\forall x\in X\right\} $ of semi-invariants of weight $m\in\mathbb{Z}$. 
\begin{notation}\label{notation:X_fl} Given a semi-invariant $f$ of weight $m\neq0$, we denote by $X_{f}$ the principal open subset of $X$ where $f$ does not vanish and, for every $\ell\geq1$,  by $\widetilde{X}_{f,\ell}$ the closed subvariety of $X\times\mathbb{G}_{m}=\mathrm{Spec}(A[t^{\pm1}])$ defined by the equation $t^{\ell}f=1$.  Note that $\widetilde{X}_{f,\ell}$ can also be seen as the closed subvariety of $X\times\mathbb{A}^{1}$ given by the same equation $t^{\ell}f=1$. 
\end{notation}

The restriction to $\widetilde{X}_{f,\ell}$ of the first projection $\mathrm{pr}_{X}$ is an \'etale Galois cover $\widetilde{X}_{f,\ell}\rightarrow X_{f}\simeq\widetilde{X}_{f,1}$ with Galois group $\mathbb{Z}_{\ell}$. On the other hand, the second projection $\mathrm{pr}_{\mathbb{G}_{m}}$ restricts on $\widetilde{X}_{f,\ell}$ to an isotrivial fibration $\rho_{f,\ell}:\widetilde{X}_{f,\ell}\rightarrow\mathbb{G}_{m}$ with fiber $F=f^{-1}(1)=\mathrm{Spec}(A/(f-1))$, which becomes trivial after the finite \'etale base change $\psi:C=\mathrm{Spec}(\mathbb{C}[v^{\pm1}])\rightarrow\mathbb{G}_{m}$, $v\mapsto v^{m}$. Indeed, since $f$ is a semi-invariant of weight $m$, the morphism  \[ F\times C\rightarrow\widetilde{X}_{f,\ell}\times_{\mathbb{G}_{m}}C,\:(x,v)\mapsto((\mu(v^{-\ell},x),v^{m}),v) \] is an isomorphism of schemes over $C$.  
\begin{prop}\label{prop:isos-explicites} With the notation above, the following hold:
\begin{enumerate}
\item If $\ell'$ is congruent to $\ell$ or $-\ell$ modulo $m$, then the fibrations $\rho_{f,\ell}:\widetilde{X}_{f,\ell}\rightarrow\mathbb{G}_{m}$ and $\rho_{f,\ell'}:\widetilde{X}_{f,\ell'}\rightarrow\mathbb{G}_{m}$ are isomorphic up to an automorphism of $\mathbb{G}_{m}$.  
\item If the residue classes modulo $m$ of $\ell$ and $\ell'$ generate the same subgroup of $\mathbb{Z}_{m}$, i.e.~if $\gcd(\ell,m)=\gcd(\ell',m)$, then $\widetilde{X}_{f,\ell}\times\mathbb{A}_{*}^{1}$ and $\widetilde{X}_{f,\ell'}\times\mathbb{A}_{*}^{1}$ are isomorphic as abstract algebraic varieties.
\end{enumerate}
\end{prop}
\begin{proof} Indeed, if  $\ell=\pm\ell'+km$ for some $k\in\Z$, then the morphism $\Phi:\widetilde{X}_{f,\ell}\rightarrow\widetilde{X}_{f,\ell'}$,  $(x,t)\mapsto(\mu(t^{k},x),t^{\pm1})$ is an isomorphism for which we have a commutative diagram  \begin{eqnarray*} \widetilde{X}_{f,\ell} & \stackrel{\Phi}{\xrightarrow{\hspace*{1cm}}} & \widetilde{X}_{f,\ell'}\\ \rho_{f,\ell}\bigg\downarrow &  & \bigg\downarrow\rho_{f,\ell'}\\ \mathbb{G}_{m} & \stackrel{t\mapsto t^{\pm1}}{\xrightarrow{\hspace*{1cm}}} & \mathbb{G}_{m}. \end{eqnarray*} For the second assertion, let us identify $\widetilde{X}_{f,\ell}\times\mathbb{A}_{*}^{1}$ and $\widetilde{X}_{f,\ell'}\times\mathbb{A}_{*}^{1}$ with the closed subvarieties of $X\times(\mathbb{A}_{*}^{1})^{2}=\mathrm{Spec}(A[t^{\pm1},u^{\pm1}])$ defined  by the equations $t^{\ell}f-1=0$ and $t^{\ell'}f-1=0$, respectively.  Let $d=\gcd(\ell,m)=\gcd(\ell',m)$. Then, there exist integers $a,b\in\Z$ satisfying $\ell'=a\ell+bm$ such that $a$ is coprime with $\frac{m}{d}$. This guarantees in turn the existence of a matrix  \[ A=\left(\begin{array}{cc} a & \frac{m}{d}\\ \alpha & \beta \end{array}\right)\in\mathrm{SL}_{2}(\mathbb{Z}) \] corresponding to an automorphism $\sigma(t,u)=(t^{a}u^{m/d},t^{\alpha}u^{\beta})$ of the torus $(\mathbb{A}_{*}^{1})^{2}=\mathrm{Spec}(\mathbb{C}[t^{\pm1},u^{\pm1}])$. Finally, a straightforward computation shows that the automorphism of $X\times(\mathbb{A}_{*}^{1})^{2}$ defined by $(x,t,u)\mapsto(\mu(t^{b}u^{-\ell/d},x),\sigma(t,u))$ maps $\widetilde{X}_{f,\ell'}\times\mathbb{A}_{*}^{1}$ isomorphically onto $\widetilde{X}_{f,\ell}\times\mathbb{A}_{*}^{1}$.\end{proof} 
\begin{rem} Note that the isomorphism constructed in the proof of  Proposition~\ref{prop:isos-explicites} does not preserve the induced isotrivial fibration $\rho_{f,\ell}\circ\mathrm{pr}_{1}:\widetilde{X}_{f,\ell}\times\mathbb{A}_{*}^{1}\rightarrow \mathbb{G}_m$ with fiber $F\times\mathbb{A}_{*}^{1}$. 
\end{rem}
In particular, observe that Proposition \ref{prop:isos-explicites} implies that if $f$ is a semi-invariant regular function of weight $m\neq0$, then the varieties $\widetilde{X}_{f,\ell}\times\mathbb{A}_{*}^{1}$ are isomorphic to  $\widetilde{X}_{f,1}\times\mathbb{A}_{*}^{1}=X_f\times\mathbb{A}_{*}^{1}$ for all integers $\ell\geq1$ relatively prime with $m$. Let us list some  consequences of this observation. 

\begin{cor} \label{cor:nonruled} Suppose that the variety $X_{f}$ is not $\mathbb{A}_{*}^{1}$-uniruled (for instance, that $X_{f}$ is smooth of log-general type). Then for all $\ell\in\mathbb{Z}_{\geq1}$ relatively prime with $m$ the varieties $\widetilde{X}_{f,\ell}$ are isomorphic. 
\end{cor}
\begin{proof} Indeed, by a strong cancellation theorem due to Iitaka and Fujita \cite{IF}  (see also \cite{Dubouloz-Tori}), every isomorphism between $\widetilde{X}_{f,\ell}\times\mathbb{A}_{*}^{1}$ and $\widetilde{X}_{f,\ell'}\times\mathbb{A}_{*}^{1}$ descends to an isomorphism between $\widetilde{X}_{f,\ell}$ and $\widetilde{X}_{f,\ell'}$. \end{proof}

\begin{cor} \label{cor:Kod1} Suppose that $X_{f}$ is a smooth factorial affine surface of logarithmic Kodaira dimension $\kappa(X_f)$ different from $1$. Then the surfaces $\widetilde{X}_{f,\ell}$ are isomorphic for all $\ell\geq1$ relatively prime with $m$ . 
\end{cor}

\begin{proof} Since  $\widetilde{X}_{f,\ell}\times\mathbb{A}_{*}^{1}$ is isomorphic to $X_{f}\times\mathbb{A}_{*}^{1}$ for all $\ell\in\mathbb{Z}_{\geq1}$ relatively prime with $m$, we deduce that $\widetilde{X}_{f,\ell}$ is smooth, factorial, of the same Kodaira dimension as $X_{f}$. So the assertion follows from Proposition 12 in \cite{Dubouloz-Tori}  which asserts that $\mathbb{A}_{*}^{1}$-cancellation holds for smooth factorial affine surfaces of logarithmic Kodaira dimension different from $1$.  \end{proof}

\subsection{Semi-invariant hypersurfaces of affine spaces}
In this subsection, we consider more specifically the case of principal open subsets $X_f\subset\mathbb{A}^{n}$  associated to semi-invariant polynomials $f\in\C[x_1,\ldots,x_n]$. As a third application of Proposition~\ref{prop:isos-explicites}, we determine  the groups $\mathcal{O}(\widetilde{X}_{f,\ell})^{\ast}$ of invertible regular functions on the varieties  $\widetilde{X}_{f,\ell}$ when $f$ is irreducible and $\ell$ is coprime with $m$. 

\begin{lemma}\label{lemma:units} Let $n,m,\ell\geq1$ be  integers,  let $f\in\C[x_1,\ldots,x_n]$ be semi-invariant of weight $m\geq1$ under an effective regular action of the multiplicative group $\mathbb{G}_{m}$ on $\C[x_1,\ldots,x_n]$ and let $\widetilde{X}_{f,\ell}\subset\mathbb{A}^{n}\times\mathbb{A}_{*}^{1}=\mathrm{Spec}(\C[x_1,\ldots,x_n][t^{\pm1}])$ be the hypersurface defined by the equation $t^{\ell}f=1$. If $f$ is irreducible and  $\ell$ and $m$ are coprime, then 
 $$\mathcal{O}(\widetilde{X}_{f,\ell})^{\ast}=\{\lambda t^i\mid \lambda\in\C^{\ast}, i\in\Z\}.$$  
\end{lemma}
\begin{proof} First  note that since the polynomial $f\in\C[x_1,\ldots,x_n]$ is irreducible, we have $$\mathcal{O}(\widetilde{X}_{f,1})^{\ast}=\left(\C[x_1,\ldots,x_n]/(tf-1)\right)^{\times}=\{\lambda t^i\mid \lambda\in\C^{\ast}, i\in\Z\}\simeq\mathbb{C}^*\cdot\mathbb{Z}.$$  By Proposition \ref{prop:isos-explicites}, if $\ell$ is coprime with $m$, then the varieties $\widetilde{X}_{f,\ell}\times\mathbb{A}^1_{\ast}$ and $\widetilde{X}_{f,1}\times\mathbb{A}^1_{\ast}$ are isomorphic and have thus isomorphic unit groups. On the other hand, denoting $\mathcal{O}(X\times\mathbb{A}^1_{\ast})=\mathcal{O}(X)[u, u^{-1}]$, we have that $\mathcal{O}(X\times\mathbb{A}^1_{\ast})^{\ast}=\{a\cdot u^i\mid a\in\mathcal{O}(X)^{\ast}, i\in\Z\}\simeq\mathcal{O}(X)^{\ast}\cdot\Z$ for every affine algebraic variety $X$. Therefore, we get that $\mathcal{O}(\widetilde{X}_{f,\ell}\times\mathbb{A}^1_{\ast})^{\ast}/\C^{\ast}\simeq\mathcal{O}(\widetilde{X}_{f,1}\times\mathbb{A}^1_{\ast})^{\ast}/\C^{\ast}\simeq\mathbb{Z}^2$ if $\ell$ is coprime with $m$, and the lemma follows.
\end{proof}

 \begin{rem}We believe that one can drop the hypothesis $\gcd(m,\ell)=1$ in the previous lemma and that the equality $\mathcal{O}(\widetilde{X}_{f,\ell})^{\ast}=\{\lambda t^i\mid \lambda\in\C^{\ast}, i\in\Z\}$ holds for all $\ell\geq1$. Nevertheless,  this seems to be a difficult question (see e.g.~Conjecture 2.9 in \cite{Ford}).
\end{rem} 

 The previous lemma turns out to be a useful tool to decide when a variety $\widetilde{X}_{f,\ell}$ is nonisomorphic to $\widetilde{X}_{f,1}$, since it allows us to reduce the problem to the study of the generic fibers of the projections  $\rho_{f,\ell}:\widetilde{X}_{f,\ell}\to\mathbb{A}_{*}^{1}$ for all $\ell$ coprime with $m$. Namely, we have the following result:

\begin{prop}\label{lemma:fibre-generique}   Let $f\in\C[x_1,\ldots,x_n]$ be irreducible and semi-invariant of weight $m\geq1$ under an effective regular action $\mu$ of the multiplicative group $\mathbb{G}_{m}$ on $\C[x_1,\ldots,x_n]$. Let $\kk=\C(t)$ and denote, for every $\ell\geq1$, by $Y_{\ell}$ the closed subvariety of $\mathbb{A}_{\kk}^{n}$ defined by the equation $f=t^{-\ell}$. Suppose that the varieties $\widetilde{X}_{f,\ell}$ and $\widetilde{X}_{f,\ell'}$ are isomorphic for some integers $\ell, \ell'$ both coprime with $m$. Let further $\ell''\geq1$ be congruent to $-\ell'$ modulo $m$. Then,  $Y_{\ell}$ is isomorphic to $Y_{\ell'}$ or to $Y_{\ell''}$.  
\end{prop}
\begin{proof}  By Lemma \ref{lemma:units}, the groups $\mathcal{O}(\widetilde{X}_{f,\ell})^{\ast}/\C^{\ast}$ and $\mathcal{O}(\widetilde{X}_{f,\ell'})^{\ast}/\C^{\ast}$ are both isomorphic to $\Z$, generated by the image of $t$. So every isomorphism $\Phi:\widetilde{X}_{f,\ell}\to \widetilde{X}_{f,\ell'}$ induces an automorphism $\varphi$ of $\mathbb{A}^1_*=\textrm{Spec}(\C[t^{\pm1}])$ of the form  $t\mapsto at^{\pm1}$ for some $a\in\C^*$. Composing $\Phi$ with the automorphism of $\widetilde{X}_{f,\ell}$ defined by  $(x,t)\mapsto(\mu(a^{\ell/m},x),a^{-1}t)$, where $a^{\ell/m}$ denotes a $m$-th root of $a^{\ell}$, we get an isomorphism $\Phi_2:\widetilde{X}_{f,\ell}\to \widetilde{X}_{f,\ell'}$ which fits into a commutative diagram  \begin{eqnarray*} \widetilde{X}_{f,\ell} & \stackrel{\Phi_2}{\xrightarrow{\hspace*{1cm}}} & \widetilde{X}_{f,\ell'}\\ \rho_{f,\ell}\bigg\downarrow &  & \bigg\downarrow\rho_{f,\ell'}\\ \textrm{Spec}(\C[t^{\pm1}]) & \stackrel{\varphi_2:t\mapsto t^{\pm1}}{\xrightarrow{\hspace*{1cm}}} & \textrm{Spec}(\C[t^{\pm1}]). \end{eqnarray*}  
If $\varphi_2(t)=t$, then we get an isomorphism between the generic fibers of $\rho_{f,\ell}$ and $\rho_{f,\ell'}$, which are isomorphic to $Y_{\ell}$ and $Y_{\ell'}$, respectively. If $\varphi_2(t)=t^{-1}$, then composing $\Phi_2$ further with the isomorphism between $\widetilde{X}_{f,\ell'}$ and $\widetilde{X}_{f,\ell''}$ defined by  $(x,t)\mapsto(\mu(t^q,x),t^{-1})$, where $q\geq1$ satisfies the equality $\ell''=-\ell'+qm$, we get an isomorphism $\Phi_3:\widetilde{X}_{f,\ell}\to \widetilde{X}_{f,\ell''}$ which fits into a commutative diagram  \begin{eqnarray*} \widetilde{X}_{f,\ell} & \stackrel{\Phi_3}{\xrightarrow{\hspace*{1cm}}} & \widetilde{X}_{f,\ell''}\\ \rho_{f,\ell}\bigg\downarrow &  & \bigg\downarrow\rho_{f,\ell''}\\ \textrm{Spec}(\C[t^{\pm1}]) & \stackrel{t\mapsto t}{\xrightarrow{\hspace*{1cm}}} & \textrm{Spec}(\C[t^{\pm1}]). \end{eqnarray*} This concludes the proof, since the above diagram implies that $Y_{\ell}$ and $Y_{\ell''}$ are isomorphic.  \end{proof}



 \section{Factorial affine surfaces failing Laurent Cancellation}
 As a first application of the previous techniques, we present new explicit examples of factorial surfaces failing Laurent cancellation. They are all realized as hypersurfaces $\widetilde{X}_{f,\ell}$ for irreducible  polynomials $f\in \mathbb{C}[x,y]$ which are semi-invariant under a faithful linear $\mathbb{G}_m$-action on $\mathbb{A}^2=\textrm{Spec}(\C[x,y])$ with positive weights. More precisely,  given positive integers $p,q\geq1$, we consider the polynomial $f=x^p+y^q$ which is semi-invariant of weight $m=pq$ under the action $\mu:\mathbb{G}_m\times\mathbb{A}^2\to\mathbb{A}^2$ defined by $\mu(\lambda,(x,y))=(\lambda^qx,\lambda^py)$.

Note that if $p$ or $q$ is equal to $1$, then the complements $X_f$ of the zero loci of  polynomials $f$  are  isomorphic to $\mathbb{A}^1\times \mathbb{A}^1_*$, and so are all the associated surfaces $\widetilde{X}_{f,\ell}$, $\ell\in\{1,\ldots,m\}$. 
Therefore, let $p,q$ be both greater than $1$, and relatively prime. Then  ${X}_{f}$ is isomorphic to the product of the smooth affine Fermat type curve $C=\{x^p+y^q=1\}\subset \mathbb{A}^2$ with $\mathbb{A}^1_*$. Recall that a smooth projective model $\overline{C}$ of $C$ has genus $g(\overline{C})=\frac{(p-1)(q-1)}{2}>0$. Consequently,  the logarithmic Kodaira dimension of $\widetilde{X}_{f,1}\simeq X_f$ is equal to $1$. Since the logarithmic Kodaira dimension is invariant under finite \'etale covers, all associated surfaces $\widetilde{X}_{f,\ell}$ are also of logarithmic Kodaira dimension  $1$. 

All surfaces $X_{f}$ obtained in this way are factorial and the associated surfaces $\widetilde{X}_{f,\ell}$ are $\mathbb{A}_{*}^{1}$-uniruled. They are in fact canonically $\mathbb{A}_{*}^{1}$-ruled over $\mathbb{A}^{1}$ by the restriction $q_{\ell}:\widetilde{X}_{f,\ell}\rightarrow\mathbb{A}^{1}$
of the rational map $V_{\ell}\dashrightarrow\mathbb{P}^{1}$ defined
by the mobile part of the divisor $K_{V_{\ell}}+B_{\ell}$ on a smooth
projective completion $V_{\ell}$ of $\widetilde{X}_{f,\ell}$ with reduced
SNC boundary $B_{\ell}$ (see e.g. \cite[Chapter 2, Section 6]{MiBook}). In view of Corollaries \ref{cor:nonruled} and \ref{cor:Kod1}
we thus expect to find counterexamples to the Laurent Cancellation Problem
among these surfaces. \\

The first case to consider is the case $(p,q)=(2,3)$. Nevertheless, it turns out to be deceptive. Indeed, on the one hand,   Proposition \ref{prop:isos-explicites} implies that every surface $\widetilde{X}_{f,\ell}$ with $f=x^{2}+y^{3}$ is isomorphic to $\widetilde{X}_{f,1}$, to $\widetilde{X}_{f,2}$ or to $\widetilde{X}_{f,3}$. On the other hand, we have the following result.

\begin{lemma}\label{lem:case23} Let $f=x^{2}+y^{3}\in\mathbb{C}[x,y]$, $n\geq1$, and let $\ell,\ell'\in\{1,2,3\}$ be distinct. Then for every $n\geq 0$, the varieties $\widetilde{X}_{f,\ell}\times(\mathbb{A}_{*}^{1})^{n}$ and $\widetilde{X}_{f,\ell'}\times(\mathbb{A}_{*}^{1})^{n}$ are not isomorphic.  
\end{lemma}

\begin{proof}  Let us identify for each $i\in\{1,2,3\}$ the surface $\widetilde{X}_{f,i}$
with the closed subvariety of $\mathrm{Spec}(\mathbb{C}[x,y,t^{\pm1}])$
defined by the equation, $t^i(x^{2}+y^{3})=1$. The log-canonical
$\mathbb{A}_{*}^{1}$-fibration $q_{i}:\widetilde{X}_{f,i}\rightarrow\mathbb{A}^{1}$
of $\widetilde{X}_{f,i}$ coincides with the algebraic quotient morphism
of the $\mathbb{G}_{m}$-action $\lambda\cdot(x,y,t)=(\lambda^{3}x,\lambda^{2}y,\lambda^{-6/i}t)$.
More concretely, we have 
\[
\begin{cases}
q_{1}:\widetilde{X}_{f,1}\rightarrow\mathbb{A}^{1}, & (x,y,v)\mapsto x^{2}t,\\
q_{2}:\widetilde{X}_{f,2}\rightarrow\mathbb{A}^{1}, & (x,y,v)\mapsto xt,\\
q_{3}:\widetilde{X}_{f,3}\rightarrow\mathbb{A}^{1}, & (x,y,v)\mapsto yt.
\end{cases}
\]
Furthermore, since $\kappa(\widetilde{X}_{f,i}\times(\mathbb{A}_{*}^{1})^{n})=\kappa(\widetilde{X}_{f,i})$,
the log-canonical fibration of $\widetilde{X}_{f,i}\times(\mathbb{A}_{*}^{1})^{n}$
coincides with $\tau_{i}=q_{i}\circ\mathrm{pr}_{\widetilde{X}_{f_{,i}}}:\widetilde{X}_{f,i}\times(\mathbb{A}_{*}^{1})^{n}\rightarrow\mathbb{A}^{1}$.
It is straightforward to check that the $\mathbb{A}_{*}^{1}$-fibrations
$q_{i}$, hence the $(\mathbb{A}_{*}^{1})^{n+1}$-fibrations $\tau_{i}$,
have different types of degenerate fibers: all their fibers are isomorphic
to $\mathbb{A}_{*}^{1}$ when equipped with their reduced structures,
but $q_{1}$ has two multiple fibers of multiplicities $2$ and
$3$, respectively, $q_{2}$ has two multiple fibers of multiplicity $3$ while
$q_{3}$ has three multiple fibers of multiplicity $2$. The lemma follows then, since
an isomorphism $\widetilde{X}_{f,i}\times(\mathbb{A}_{*}^{1})^{n}$ and
$\tilde{X}_{f,j}\times(\mathbb{A}_{*}^{1})^{n}$, $i,j=1,2,3$, must
be compatible with these log-canonical fibrations.
\end{proof}

\noindent The next simplest case, namely the case $(p,q)=(2,5)$, does
lead to counterexamples.

\begin{prop} \label{prop:case25} The hypersurfaces $X_1, X_3\subset\mathbb{A}^3=\mathrm{Spec}(\mathbb{C}[x,y,t])$ defined by the equation $t(x^{2}+y^{5})=1$ and $t^3(x^{2}+y^{5})=1$, respectively, are counterexamples to the Laurent Cancellation Problem. 
\end{prop}

\begin{proof}
Consider the polynomial $f=x^{2}+y^{5}\in\mathbb{C}[x,y]$ and observe that $X_1$ and $X_3$ are isomorphic to the hypersurfaces $\widetilde{X}_{f,1}$ and $\widetilde{X}_{f,3}$, respectively. Since $f$ is  semi-invariant of weight $10$ for the linear $\mathbb{G}_{m}$-action on $\mathbb{A}^{2}$ with weights $(5,2)$, the $\mathbb{A}^1_*$-cylinders $X_1\times\mathbb{A}^1_*$ and $X_3\times\mathbb{A}^1_*$ are isomorphic by Proposition~\ref{prop:isos-explicites}. To show that $X_1$ and $X_3$ are not isomorphic, it is enough, by virtue of  Proposition  \ref{lemma:fibre-generique}, to check that the curve $Y_1\subset\mathbb{A}_{\C(t)}^{2}=\textrm{Spec}(\C(t)[x,y])$ given by the equation $x^2+y^5=t^{-1}$ is isomorphic over $\mathbb{C}(t)$ neither  to the curve $Y_3$ of equation  $x^2+y^5=t^{-3}$, nor to that $Y_7$ of equation $x^2+y^5=t^{-7}$. This can be seen using elementary geometric arguments as follows.

A $\mathbb{C}(t)$-isomorphism $\psi:Y_{1}\stackrel{\sim}{\rightarrow}Y_{i}$
between $Y_{1}$ and $Y_{i}$, $i=3,7$, must (if it exists) extend to an isomorphism
$\Psi:\overline{Y}_{1}\stackrel{\sim}{\rightarrow}\overline{Y}_{i}$
between the respective normalizations $\overline{Y}_{1}$ and $\overline{Y}_{i}$
of the projective closures of $Y_{1}$ and $Y_{i}$ in $\mathbb{P}_{\mathbb{C}(t)}^{2}=\mathrm{Proj}(\mathbb{C}(t)[x,y,z])$.
A direct computation shows that all these curves have genus $2$ and
that their respective canonical degree $2$ covers $\overline{\pi}_{i}:\overline{Y}_{j}\rightarrow\mathbb{P}_{\mathbb{C}(v)}^{1}$
defined by the complete linear systems $|K_{\overline{Y}_{j}}|$,
$j=1,3,7$, coincide with the maps induced by the projection from
the point $[1:0:0]$ in $\mathbb{P}_{\mathbb{C}(t)}^{2}$. The restriction
of $\overline{\pi}_{i}$ to $Y_{i}$ coincides with the projection
$\pi_{i}=\mathrm{pr}_{y}:Y_{i}=\{x^{2}+y^{5}-t^{-i}=0\}\rightarrow\mathbb{A}_{\mathbb{C}(t)}^{1}=\mathrm{Spec}(\mathbb{C}(t)[y])$,
and since $\Psi^{*}K_{\overline{Y}_{i}}=K_{\overline{Y}_{1}}$, it
follows that there exists a $\mathbb{C}(t)$-automorphism $\theta$
of $\mathbb{A}_{\mathbb{C}(t)}^{1}$ for which the following diagram
commutes 
\begin{eqnarray*}
Y_{1} & \stackrel{\psi}{\xrightarrow{\hspace*{1cm}}} & Y_{i}\\
\pi_{1}\bigg\downarrow &  & \bigg\downarrow\pi_{3}\\
\mathbb{A}_{\mathbb{C}(t)}^{1} & \stackrel{\theta}{\xrightarrow{\hspace*{1cm}}} & \mathbb{A}_{\mathbb{C}(t)}^{1}.
\end{eqnarray*}

So $\theta$ is an affine transformation of the form $y\mapsto a(t)y+b(t)$
for some pair $(a(t),b(t))\in\mathbb{C}(t)^{*}\times\mathbb{C}(t)$
which maps the branch locus $B_{1}=\left\{ y^{5}-t^{-1}=0\right\}$
of $\pi_{1}$ isomorphically onto that $B_{i}=\{y^{5}-t^{-i}=0\}$
of $\pi_{i}$, $i=3,7$. Thus $b(t)=0$ necessarily and $a(t)^{5}=t^{-i+1}$,
which is absurd since neither $t^{-2}$ nor $t^{-6}$ admits a fifth
root in $\mathbb{C}(t)$. This shows that $X_{1}$ and $X_{3}$ are
not isomorphic. 
\end{proof}

\begin{rem} The above surfaces $X_{1}$ and $X_{3}$ are precisely
those constructed in a more geometric fashion in \cite[Subsection 2.2]{Dubouloz-Tori}.
This can be seen by considering  the structure of their
log-canonical $\mathbb{A}_{*}^{1}$-fibrations. With the same notation as in the proof of Proposition \ref{prop:case25}, these fibrations coincide  with the
morphisms 
\begin{eqnarray*}
q_{1}:X_{1}\rightarrow\mathbb{A}^{1},\;(x,y,t)\mapsto y^{5}t & \textrm{and} & q_{3}:X_{3}\rightarrow\mathbb{A}^{1},\;(x,y,t)\mapsto y^{5}t^{3}.
\end{eqnarray*}
Note that for the same reason as in the proof of Lemma~\ref{lem:case23}, every
isomorphism $\Psi$ between $X_{3}\times\mathbb{A}_{*}^{1}$ and $X_{1}\times\mathbb{A}_{*}^{1}$
must be compatible with the log-canonical $(\mathbb{A}_{*}^{1})^{2}$-fibrations
$q_{1}\circ\mathrm{pr}_{1}:X_{1}\times\mathbb{A}_{*}^{1}\rightarrow\mathbb{A}^{1}$
and $q_{3}\circ\mathrm{pr}_{1}:X_{3}\times\mathbb{A}_{*}^{1}\rightarrow\mathbb{A}^{1}$.

This does of course hold for the explicit isomorphism $\Psi:X_{3}\times\mathbb{A}_{*}^{1}\stackrel{\sim}{\rightarrow}X_{1}\times\mathbb{A}_{*}^{1}$
constructed via the procedure described in the proof of Proposition
\ref{prop:isos-explicites}, which is obtained as follows. With the notation of this proof, we have in our case $\ell=1$, $\ell'=3$, $m=10$, $d=1$, and we can take $a=3$, $b=0$, $\alpha=-1$ and $\beta=-3$. This gives the isomorphism  $\Psi:(x,y,t,u)\mapsto(u^{-5}x,u^{-2}y,t^{3}u^{10},t^{-1}u^{-3})$.
One thus has  $\Psi^{*}(y^{5}t)=(u^{-10}y^{5})(t^{3}u^{10})=y^{5}t^{3}$, 
hence a commutative diagram 
\begin{eqnarray*}
X_{3}\times\mathbb{A}_{*}^{1} & \stackrel{\Psi}{\xrightarrow{\hspace*{1cm}}} & X_{1}\times\mathbb{A}_{*}^{1}\\
 q_{3}\circ\mathrm{pr}_{1}\bigg\downarrow &  & \bigg\downarrow q_{1}\circ\mathrm{pr}_{1}\\
\mathbb{A}^{1} & = & \mathbb{A}^{1}.
\end{eqnarray*}
\end{rem}

Let us emphasize two ingredients which played a crucial role in the proof of the fact that the above surfaces $X_{1}$ and $X_{3}$ are not isomorphic. The first one is that the curves $Y_{i}$, $i=1,3,7$, are cyclic covers of $\mathbb{A}_{\mathbb{C}(t)}^{1}$. The second one is that  every isomorphism between them descends to an automorphism
of $\mathbb{A}_{\mathbb{C}(t)}^{1}$. Generalizing this type of arguments allows us to obtain the following infinite
families of counterexamples.

\begin{thm} Let $p\geq2$ be a prime number, let $q\geq3$ be an
integer relatively prime with $p$ and let $f=x^{p}+y^{q}\in\mathbb{C}[x,y]$.
Then for every $\ell\geq2$ relatively prime with $m=pq$, the surfaces
$\widetilde{X}_{f,1}$ and $\widetilde{X}_{f,\ell}$ are nonisomorphic, with
isomorphic $\mathbb{A}_{*}^{1}$-cylinders. 

\end{thm}
\begin{proof}
The polynomial $f$ being an irreducible semi-invariant of weight
$m$ for the linear action of $\mathbb{G}_{m}$ on $\mathbb{A}^{2}$
with weight $(q,p)$, the fact that $X_{1}=\widetilde{X}_{f,1}$ and $X_{\ell}=\widetilde{X}_{f,\ell}$
have isomorphic $\mathbb{A}_{*}^{1}$-cylinders follows again from
Proposition \ref{prop:isos-explicites}. 

On the other hand, by virtue of Proposition
\ref{lemma:fibre-generique}, $X_{1}$ and $X_{\ell}$ are nonisomorphic
provided that the curve $Y_{1}=\{x^{p}+y^{q}-t^{-1}=0\}$ in $\mathbb{A}_{\mathbb{C}(t)}^{2}$
is isomorphic over $\mathbb{C}(t)$ neither to $Y_{\ell}=\{x^{p}+y^{q}-t^{-\ell}=0\}$
nor to $Y_{\ell''}=\{x^{p}+y^{q}-t^{-\ell''}=0\}$ for some integer
$\ell''\geq1$ congruent to $-\ell'$ modulo $m$. This can be proved in exactly the same way as for Proposition \ref{prop:case25} as soon as we show that, if it exists, a $\mathbb{C}(t)$-isomorphism
$\psi:Y_{1}\stackrel{\sim}{\rightarrow}Y_{j}$, $j=\ell,\ell''$ 
descends to a $\mathbb{C}(t)$-automorphism $\theta$ of $\mathbb{A}_{\mathbb{C}(t)}^{1}$
making the following diagram commutative 
\begin{eqnarray*}
Y_{1} & \stackrel{\psi}{\xrightarrow{\hspace*{1cm}}} & Y_{j}\\
\pi_{1}=\mathrm{pr}_{y}\bigg\downarrow &  & \bigg\downarrow\pi_{j}=\mathrm{pr}_{y}\\
\mathbb{A}_{\mathbb{C}(t)}^{1} & \stackrel{\theta}{\xrightarrow{\hspace*{1cm}}} & \mathbb{A}_{\mathbb{C}(t)}^{1}.
\end{eqnarray*}

Again, every $\mathbb{C}(t)$-isomorphism $\psi:Y_{1}\stackrel{\sim}{\rightarrow}Y_{j}$
uniquely extends to a $\mathbb{C}(t)$-isomorphism $\Psi:\overline{Y}_{1}\stackrel{\sim}{\rightarrow}\overline{Y}_{j}$
between the normalizations $\overline{Y}_{i}$ of the respective projective
closures of the curves $Y_{i}$ in $\mathbb{P}_{\mathbb{C}(t)}^{2}$.
The latter are smooth curves of genus $g=\frac{(p-1)(q-1)}{2}\geq2$
on which $\pi_{i}$ extends to a cyclic Galois cover $\overline{\pi}_{i}:\overline{Y}_{i}\rightarrow\mathbb{P}_{\mathbb{C}(t)}^{1}$
of prime order $p$, which is totally ramified over the $\mathbb{C}(t)$-rational
point $\mathbb{P}_{\mathbb{C}(t)}^{1}\setminus\mathbb{A}_{\mathbb{C}(t)}^{1}$.
Base changing to the algebraic closure $K$ of $\mathbb{C}(t)$, it
follows from \cite[Theorem 1]{Gonzalez} and  \cite[Main Theorem]{Nakajo}, which are actually
stated for complex curves but whose proofs carry on verbatim to smooth
curves defined over an algebraically closed field of characteristic
zero, that $\Psi_{K}:\overline{Y}_{1,K}\stackrel{\sim}{\rightarrow}\overline{Y}_{j,K}$
descends to a $K$-automorphism $\overline{\theta}_{K}:\mathbb{P}_{K}^{1}\stackrel{\sim}{\rightarrow}\mathbb{P}_{K}^{1}$ that maps the branch locus of $\overline{\pi}_{1,K}$ isomorphically
onto that of $\overline{\pi}_{j,K}$. Furthermore, since $q\geq2$,
the branch locus of $\overline{\pi}_{i,K}$ consists of at least three
points, implying that the $K$-automorphism $\overline{\theta}_{K}$
for which $\overline{\pi}_{j,K}\circ\Psi_{K}=\overline{\theta}_{K}\circ\overline{\pi}_{1,K}$
is unique, hence descends to a $\mathbb{C}(t)$-automorphism $\overline{\theta}$
such that $\overline{\pi}_{j}\circ\Psi=\overline{\theta}\circ\overline{\pi}_{1}$.
Since $\overline{Y}_{i}\setminus Y_{i}$ consists of a unique $\mathbb{C}(t)$-rational
point, say $\infty_{i}$, and since $\Psi(\infty_{1})=\infty_{j}$,
$\overline{\theta}$ restricts to an isomorphism $\theta:\mathbb{A}_{\mathbb{C}(t)}^{1}=\mathbb{P}_{\mathbb{C}(t)}^{1}\setminus\overline{\pi}_{1}(\infty_{1})\stackrel{\sim}{\rightarrow}\mathbb{A}_{\mathbb{C}(t)}^{1}=\mathbb{P}_{\mathbb{C}(t)}^{1}\setminus\overline{\pi}_{j}(\infty_{j})$
for which the diagram above commutes, as desired. 
\end{proof}

 \section{On the case of negative logarithmic Kodaira dimension} 

The aim of this section is to construct explicit examples of smooth affine varieties of negative logarithmic Kodaira
dimension  failing  the Laurent Cancellation property. Let us first remark that such examples must be of dimension at least three. Indeed, we can extend a result of \cite{Dubouloz-Tori}, that states that  Laurent Cancellation does hold
for smooth factorial affine surfaces of negative logarithmic Kodaira
dimension, to the case  of  arbitrary smooth affine
surfaces of negative logarithmic Kodaira dimension.

\begin{thm}
Two smooth affine surfaces $S$ and $S'$ of negative logarithmic Kodaira dimension
have isomorphic $\mathbb{A}_{*}^{1}$-cylinders $S\times\mathbb{A}_{*}^{1}$
and $S'\times\mathbb{A}_{*}^{1}$ if and only if they are isomorphic. 
\end{thm}

\begin{proof}
By \cite{MiSu80}, a smooth affine surface of negative logarithmic Kodaira dimension
admits a faithfully flat morphism $\rho:S\rightarrow B$ over a smooth
curve $B$, with generic fiber isomorphic to the affine line over
the function field of $B$, called an $\mathbb{A}^{1}$-fibration.
Thus $\Gamma(S,\mathcal{O}_{S}^{*})=\rho^{*}\Gamma(B,\mathcal{O}_{B}^{*})$
and combined with the fact that cancellation holds when all invertible
functions on $S$ or $S'$ are constant (see \cite{Freudenburg,Dubouloz-Tori}),
we can restrict from now on to the case where $S$ and $S'$ admit
$\mathbb{A}^{1}$-fibrations $\rho:S\rightarrow B$ and $\rho':S'\rightarrow B'$
over smooth curves $B$ and $B'$ respectively, admitting non constant
invertible functions. In particular $B$ and $B'$ are affine, of
nonnegative logarithmic Kodaira dimension, implying that every morphism from $\mathbb{A}^{1}$
to $B\times\mathbb{A}_{*}^{1}$ or $B'\times\mathbb{A}_{*}^{1}$ is
constant. Since all irreducible components of fibers of $\rho:S\rightarrow B$
over closed points of $B$ are isomorphic to $\mathbb{A}^{1}$, it
follows that every isomorphism $\Psi:S\times\mathbb{A}_{*}^{1}\stackrel{\sim}{\rightarrow}S'\times\mathbb{A}_{*}^{1}$
descends to an isomorphism $\psi:B\times\mathbb{A}_{*}^{1}\rightarrow B'\times\mathbb{A}_{*}^{1}$
making the following diagram commutative 
\begin{eqnarray*}
S\times\mathbb{A}_{*}^{1} & \stackrel{\Psi}{\xrightarrow{\hspace*{1cm}}} & S'\times\mathbb{A}_{*}^{1}\\
\pi=(\rho,\mathrm{pr}_{2})\bigg\downarrow &  & \bigg\downarrow\pi'=(\rho',\mathrm{pr}_{2})\\
B\times\mathbb{A}_{*}^{1} & \stackrel{\psi}{\xrightarrow{\hspace*{1cm}}} & B'\times\mathbb{A}_{*}^{1}.
\end{eqnarray*}
If either $\kappa(B)\neq0$ or $\kappa(B')\neq0$ then, by Iitaka-Fujita
strong cancellation Theorem \cite{IF}, $\psi$ descends further to
an isomorphism $\overline{\psi}:B\rightarrow B'$ such that $\mathrm{pr}_{B'}\circ\psi=\overline{\psi}\circ\mathrm{pr}_{B}$.
In the case where $\kappa(B)=\kappa(B')=0$, $B$ and $B'$ are both
isomorphic to $\mathbb{A}_{*}^{1}=\mathrm{Spec}(\mathbb{C}[u^{\pm1}])$
and we have the following alternative: either $\rho:S\rightarrow B$
is a locally trivial, hence trivial, $\mathbb{A}^{1}$-bundle and
then so is $\rho':S'\rightarrow B'$ by the commutativity of the above
diagram or $\rho:S\rightarrow B$ has at least a degenerate fiber.
In the second case, letting $b_{1},\ldots,b_{s}\in B$ and $b_{1}',\ldots,b_{s'}'\in B'$
be the closed points over which the fibers or $\rho$ and $\rho'$
respectively are degenerate, the morphisms $\pi$ and $\pi'$ degenerate
respectively over the sections $\{b_{i}\}\times\mathbb{A}_{*}^{1}$
and $\{b_{i}'\}\times\mathbb{A}_{*}^{1}$ of the projections $\mathrm{pr}_{\mathbb{A}_{*}^{1}}:B\times\mathbb{A}_{*}^{1}\rightarrow\mathbb{A}_{*}^{1}$
and $\mathrm{pr}_{\mathbb{A}_{*}^{1}}:B'\times\mathbb{A}_{*}^{1}\rightarrow\mathbb{A}_{*}^{1}$.
Identifying $B\times\mathbb{A}_{*}^{1}$ and $B'\times\mathbb{A}_{*}^{1}$
with the torus $\mathbf{T}^{2}=\mathrm{Spec}(\mathbb{C}[u^{\pm1},t^{\pm1}])$,
the automorphism $\psi$ has the form $(u,t)\mapsto(\lambda_{1}u^{\alpha}t^{\beta},\lambda_{2}u^{\gamma}t^{\delta})$
where $\lambda_{i}\in\mathbb{C}^{*}$ and $\left(\begin{array}{cc}
\alpha & \beta\\
\gamma & \delta
\end{array}\right)\in\mathrm{GL}_{2}(\mathbb{Z})$. The condition $\psi(\{b_{i}\}\times\mathbb{A}_{*}^{1})=\{b_{j(i)}'\}\times\mathbb{A}_{*}^{1}$,
$i=1,\ldots,s$, $j(i)\in\{1,\ldots,s'\}$ implies that $\beta=0$
hence that $\alpha\delta=\pm1$. It follows that $\psi$ descends
to an isomorphism $\overline{\psi}:B\rightarrow B'$ of the form $u\mapsto\lambda_{1}u^{\pm1}$.
Summing up, either $S$ and $S'$ are both isomorphic to the trivial
$\mathbb{A}^{1}$-bundle over $\mathbb{A}_{*}^{1}$ and we are done
already, or we have a commutative diagram

\begin{eqnarray*}
S\times\mathbb{A}_{*}^{1} & \stackrel{\Psi}{\xrightarrow{\hspace*{1cm}}} & S'\times\mathbb{A}_{*}^{1}\\
\pi=(\rho,\mathrm{pr}_{2})\bigg\downarrow &  & \bigg\downarrow\pi'=(\rho',\mathrm{pr}_{2})\\
B\times\mathbb{A}_{*}^{1} & \stackrel{\psi}{\xrightarrow{\hspace*{1cm}}} & B'\times\mathbb{A}_{*}^{1}\\
\mathrm{pr}_{B}\bigg\downarrow &  & \bigg\downarrow\mathrm{pr}_{B'}\\
B & \stackrel{\overline{\psi}}{\xrightarrow{\hspace*{1cm}}} & B'
\end{eqnarray*}
for some isomorphism $\overline{\psi}:B\stackrel{\sim}{\rightarrow}B'$.
Replacing $S'$ by the isomorphic surface $S'\times_{B'}B$, we may
assume further that $B'=B$ and that $\overline{\psi}=\mathrm{id}_{B}$.
The commutativity of the above diagram then implies that $\psi$ maps
the section $C=B\times\{1\}\subset B\times\mathbb{A}_{*}^{1}$ isomorphically
onto a section $C'\subset B'\times\mathbb{A}_{*}^{1}$ or $\mathrm{pr}_{B'}$
while $\Psi$ maps $\pi^{-1}(C)$ isomorphically onto ${\pi'}^{-1}(C')$.
The assertion follows since $\pi^{-1}(C)$ and ${\pi'}^{-1}(C')$
are isomorphic to $S$ and $S'$ respectively. 
\end{proof}

From now on, we will use the following notation.

\begin{notation} Given integers $n\geq1$ and $m\geq2$, we  denote  by $f_{n,m}$ the polynomial $$f_{n,m}=(\prod_{i=1}^{n}x_{i}^{2})z-y^{m}\in\mathbb{C}[x_{1},\ldots,x_{n}][y,z].$$ \noindent Since  $f_{n,m}$  is a semi-invariant of weight $m$ for the linear $\mathbb{G}_{m}$-action on $\mathbb{A}^{n+2}=\mathrm{Spec}(\mathbb{C}[x_{1},\ldots,x_{n}][y,z])$ with weights $(1,\ldots,1,1,m-2n)$, we will follow the notation of the previous sections and consider, for every $\ell\in\mathbb{Z}_{\geq1}$, the hypersurface $X_{n,m,\ell}=\widetilde{X}_{f_{n,m},\ell}$ of $\mathbb{A}^{n+3}=\mathrm{Spec}(\mathbb{C}[x_{1},\ldots,x_{n}][y,z,t])$ defined by the equation $$t^{\ell}\left(x_{1}^{2}\cdots x_{n}^{2} z-y^{m}\right)=1.$$ \end{notation}

\begin{lemma}\label{lemma:factorial}
Every hypersurface $X_{n,m,\ell}$ defined above is smooth and has negative logarithmic Kodaira dimension. Moreover, $X_{n,m,\ell}$ is factorial provided that $\ell$ is relatively prime with $m$.
\end{lemma}

\begin{proof}
Consider the open subset $\mathcal{U}=X_{n,m,\ell}\setminus\{x_1\cdots x_n=0\}$ of $X_{n,m,\ell}$. Since 
$$\mathcal{U}=\{(x_1,\ldots,x_n,y,z,t)\in\mathbb{A}^{n+3}\mid x_1\cdots x_nt\neq0, z=x_1^{-2}\cdots x_n^{-2}(t^{-\ell}+y^m)\}\simeq(\mathbb{A}_{*}^{1})^{n+1}\times\mathbb{A}^{1},$$
 it follows that $\mathcal{U}$ has negative logarithmic Kodaira dimension, hence that $X_{n,m,\ell}$ has negative logarithmic Kodaira dimension too.

From now on, we suppose that $\ell$ is coprime with $m$.  Then, since $$\mathcal{O}(X_{n,m,\ell})/(x_n)\simeq\C[x_1,\ldots,x_{n-1},y,z,t]/(t^ly^m+1)$$ is an integral domain, when  $\ell$ is coprime with $m$, it follows that $(x_n)$ is a prime ideal of  $\mathcal{O}(X_{n,m,\ell})$. 
We further consider the localization $\mathcal{O}(X_{n,m,\ell})_{x_n}$ of $\mathcal{O}(X_{n,m,\ell})$ with respect to this ideal $(x_n)$. If $n=1$, we can express  $z=x_1^{-2}(y^m+t^{-\ell})$ and we thus get that $\mathcal{O}(X_{n,m,\ell})_{x_1}\simeq\C[y,x_1^{\pm1},t^{\pm1}]$ is factorial. By a well-known result of Nagata (see e.g.~\cite{Samuel-UFD}), this implies that $\mathcal{O}(X_{n,m,\ell})$ is factorial for $n=1$. By induction on $n$, it follows that 
\[\mathcal{O}(X_{n,m,\ell})_{x_n}\simeq\C[x_n^{\pm1}][x_1,\ldots,x_{n-1},y,z,t]/(t^{\ell}(\prod_{i=1}^{n-1}x_{i}^{2}z-y^{m})-1)\]
 is factorial for all $n\geq1$, and so Nagata's result allows us  to conclude that $\mathcal{O}(X_{n,m,\ell})$ is factorial.
\end{proof}

The next theorem is the main result of this section. It shows that we can find counterexamples to the Laurent Cancellation Problem among the  above varieties $X_{n,m,\ell}$.

\begin{thm}\label{main-thm}
Let $n\geq1$, $m\geq2$ and let $\ell,\ell'\geq1$ relatively prime with $m$ be such that $\ell'$
is not congruent to $\pm\ell$ modulo $m$. Then the hypersurfaces $$X_{n,m,\ell}=\{t^{\ell}\left(x_{1}^{2}\cdots x_{n}^{2} z-y^{m}\right)=1\}\subset\mathbb{A}^{n+3}$$
and $$X_{n,m,\ell'}=\{t^{\ell'}\left(x_{1}^{2}\cdots x_{n}^{2} z-y^{m}\right)=1\}$$ are nonisomorphic factorial affine
varieties of dimension $d=n+2$ with isomorphic $\mathbb{A}_{*}^{1}$-cylinders
$$X_{n,m,\ell}\times\mathbb{A}_{*}^{1}\simeq X_{n,m,\ell'}\times\mathbb{A}_{*}^{1}.$$ \end{thm}

\begin{proof}On the one hand, the $\mathbb{A}_{*}^{1}$-cylinders $X_{n,m,\ell}\times\mathbb{A}_{*}^{1}$ and $X_{n,m,\ell'}\times\mathbb{A}_{*}^{1}$ are isomorphic by Proposition \ref{prop:isos-explicites}. On the other hand, the first assertion of the theorem  follows directly from Lemma~\ref{lemma:fibre-generique} and Proposition \ref{prop:Danielewski} below. Indeed, there exists an element $\alpha\in\kk=\C(t)$ such that $t^{-\ell}=\alpha^mt^{-(\pm\ell')}$ if and only if $\ell$ is congruent to $\pm\ell'$ modulo $m$.
\end{proof}

\begin{prop}\label{prop:Danielewski}
Let $n\geq1$ and $m\geq2$ be  integers and let $\mathbf{k}$ be a field of  characteristic zero. Consider, for every $\lambda\in\kk$, the ring 
\[B_{\lambda}=\kk[x_1,\ldots,x_n,y,z]/(x_1^2\cdots x_n^2z-y^m-\lambda).\]
Two such rings $B_{\lambda}$ and $B_{\lambda'}$ are isomorphic if and only if there exists a constant $\alpha\in\kk^*$ such that $\lambda=\alpha^m\lambda'$.
\end{prop}

\begin{proof}If $\lambda=\alpha^m\lambda'$ for some $\alpha\in\kk^*$, then it is obvious that $B_{\lambda}$ and $B_{\lambda'}$ are isomorphic via a linear change of coordinates. It remains to prove the converse implication. For this, we use techniques from  the theory of locally nilpotent derivations, that were mainly developed by Makar-Limanov and  became progressively classical tools in affine algebraic geometry. Actually, our proof simply recollects arguments that were already given in \cite{Dubouloz} and \cite{Poloni}.   

 Let $\delta$ be a nonzero locally nilpotent derivations on $B_{\lambda}$. From \cite[Proposition 2.2 and Corollary 2.1]{Dubouloz}, which remain valid over any field of  characteristic zero, we have that  $\Ker(\delta)=\C[x_1,\ldots,x_n]$ and $\Ker(\delta^2)\subset\C[x_1,\ldots,x_n,y]$ both hold. Then, arguing exactly as in \cite[Proposition 2.3]{Poloni}, it follows that  $\Ker(\delta^2)=\C[x_1,\ldots,x_n]y+\C[x_1,\ldots,x_n]$ and that 
\[\delta=h(x_1,\ldots,x_n)\left(x_1^2\cdots x_n^2\frac{\partial}{\partial y}+my^{m-1}\frac{\partial}{\partial z}\right),\]
for some $h(x_1,\ldots,x_n)\in\kk[x_1,\ldots,x_n]$.

Now, let $\varphi:B_{\lambda}\to B_{\lambda'}$ be an isomorphism and denote by $x_1,\ldots,x_n,y,z$ (resp.~$x'_1,\ldots,x'_n,y',z'$) the images of $x_1,\ldots,x_n,y,z$ in $B_{\lambda}$ (resp.~in $B_{\lambda'}$). Similarly as in \cite[Proposition 2.5]{Poloni} we can infer from the above properties that there exist nonzero constants $a_1,\ldots,a_n\in\kk^*$, $\alpha\in\kk^*$, a polynomial $\beta\in\kk[X_1,\ldots,X_n]$ and a bijection $\sigma$ of the set $\{1,\ldots,n\}$ such that $\varphi(x_i)=a_i\varphi(x_{\sigma(i)}')$ for all $1\leq i\leq n$ and such that $\varphi(y)=\alpha y'+\beta(x_1',\ldots,x_n')$.

Finally, we write that 
\[\varphi(x_1^2\cdots x_n^2z-y^m-\lambda)=(\prod_{i=1}^na_i)^2(x_1'\cdots x_n')^2\varphi(z)-(\alpha y'+\beta(x_1',\ldots,x_n'))^m-\lambda\]
 is equal to zero in $B_{\lambda'}$. This gives the existence of  polynomials $R$ and $S$ in $\kk[X_1,\ldots,X_n,Y,Z]$ such that
\[(\prod_{i=1}^na_i)^2(X_1\cdots X_n)^2R-(\alpha Y+\beta(X_1,\ldots,X_n))^m-\lambda=S\cdot(X_1^2\cdots X_n^2Z-Y^m-\lambda').\]
From this, we deduce that
\[-(\alpha Y+\beta(0,\ldots,0))^m-\lambda=S(0,\ldots,0,Y,0)\cdot(-Y^m-\lambda'),\]
which implies that $S(0,\ldots,0,Y,0)$ is in fact a nonzero constant. Therefore, we have $\beta(0,\ldots,0)=0$,   $S(0,\ldots,0,Y,0)=\alpha^m$ and thus $\lambda=S(0,\ldots,0,Y,0)\lambda'=\alpha^m\lambda'$, as desired.
\end{proof}

\begin{example}
The simplest case in Theorem \ref{main-thm} holds when $n=1$, $m=5$, $\ell=1$ and $\ell'=2$, in which case we obtain 
that the factorial threefolds in $\mathbb{A}^{4}=\mathrm{Spec}(\mathbb{C}[x,y,z,t])$
defined by the equations $t(x^{2}z-y^5)=1$ and $t^{2}(x^{2}z-y^5)=1$, respectively,
are nonisomorphic but have isomorphic $\mathbb{A}_{*}^{1}$-cylinders.
\end{example}


\section{Affine varieties with nonisomorphic square roots}

In this section, we construct nonisomorphic affine algebraic varieties $X$ and $Y$ with isomorphic cartesian products $X\times X$ and $Y\times Y$. To begin with, we adapt Proposition~\ref{prop:isos-explicites} to obtain isomorphic cartesian products. 

\begin{lemma}\label{lem:squares} Let $n\geq1$ and let $f$ (resp.~$g$) be a regular function on the affine $n$-space $X=\mathbb{A}^n$ which is semi-invariant of weight $m\neq0$ for some action $\mu:\mathbb{G}_{m}\times X\rightarrow X$ (resp.~$\nu:\mathbb{G}_{m}\times X\rightarrow X$) of the multiplicative group on $X$. Let $a,b\geq1$ be integers such that $ab$ is congruent to $1$ modulo $m^2$. Then the products $\widetilde{X}_{f,\ell}\times\widetilde{X}_{g,\ell'}$ and $\widetilde{X}_{f,a\ell}\times\widetilde{X}_{g,b\ell'}$ are isomorphic varieties for all integers $\ell,\ell'\geq1$.
\end{lemma}

\begin{proof} Let us identify the products $\Pi_1=\widetilde{X}_{f,\ell}\times\widetilde{X}_{g,\ell'}$ and $\Pi_2=\widetilde{X}_{f,a\ell}\times\widetilde{X}_{g,b\ell'}$ with the closed subvarieties of $(X\times\mathbb{A}_{*}^{1})^2=\mathrm{Spec}(\C[x_1\ldots,x_n,y_1,\ldots,y_n][t^{\pm1},s^{\pm1}])$
defined  by the ideals 
$$I_1=\left(t^{\ell}f(x_1,\ldots,x_n)-1, s^{\ell'}g(y_1,\ldots,y_n)-1\right)$$ 
and 
$$I_2=\left(t^{a\ell}f(x_1,\ldots,x_n)-1, s^{b\ell'}g(y_1,\ldots,y_n)-1\right),$$ respectively.

Since $ab$ is congruent to $1$ modulo $m^2$, there exists an integer $c\in\Z$ such that the matrix $\begin{pmatrix}a&cm\\m&b\end{pmatrix}$ belongs to $\mathrm{SL}_{2}(\mathbb{Z})$. This matrix corresponds to an automorphism $\sigma(t,s)=(t^{a}s^{cm},t^{m}s^{b})$ of the torus $(\mathbb{A}_{*}^{1})^{2}=\mathrm{Spec}(\mathbb{C}[t^{\pm1},s^{\pm1}])$. Then, it suffices to remark that the  automorphism of $(X\times\mathbb{A}_{*}^{1})^{2}$
defined by $(x,y,t,s)\mapsto(\mu(s^{-c\ell},x), \nu(t^{-\ell'},y), \sigma(t,s))$ maps
$\Pi_2$ isomorphically onto
$\Pi_1$.
\end{proof}

\begin{cor}\label{cor:squares}
Let $X=\mathbb{A}^n$ and let $f\in\mathcal{O}(X)$ be semi-invariant of weight $m\neq0$ for a $\mathbb{G}_{m}$-action   on $X$. Suppose that there exist integers $a,b\geq1$ satisfying the two following congruences 
\[a+b\equiv0\mod(m)\quad\text{and}\quad ab\equiv1\mod(m^2).\]
Then, the varieties $\widetilde{X}_{f,1}\times\widetilde{X}_{f,1}$ and  $\widetilde{X}_{f,a}\times\widetilde{X}_{f,a}$ are isomorphic.
\end{cor}

\begin{proof}
On one hand, we have that $\widetilde{X}_{f,1}\times\widetilde{X}_{f,1}$ is isomorphic to $\widetilde{X}_{f,a}\times\widetilde{X}_{f,b}$ by Lemma \ref{lem:squares}. On the other hand, it follows from Proposition \ref{prop:isos-explicites} that $\widetilde{X}_{f,a}$ and $\widetilde{X}_{f,b}$ are isomorphic. 
\end{proof}

Combining the above corollary with the results of the previous sections, we obtain examples of nonisomorphic affine varieties whose square are isomorphic. The simplest case occurs for $m=5$. Let us denote, as in Theorem~\ref{main-thm}, by $X_{n,m,\ell}$ the hypersurface of $\mathbb{A}^{n+3}$ defined by the equation $t^{\ell}\left(x_{1}^{2}\cdots x_{n}^{2} z-y^{m}\right)=1$.

\begin{prop}  The varieties  $X=X_{n,5,1}$ and $Y=X_{n,5,2}$ are not isomorphic, although their squares $X\times X$ and $Y\times Y$ are isomorphic.
\end{prop}

\begin{proof} We deduce that $X\times X$ and $Y\times Y$ are isomorphic from  Corollary \ref{cor:squares} with  $m=5$, $a=2$ and $b=13$. The fact that $X$ and $Y$ are not isomorphic is a particular case of Theorem~\ref{main-thm}.
\end{proof}

Corollary \ref{cor:squares} allows us to find also two-dimensional examples. In particular, the two  nonisomorphic surfaces of Proposition \ref{prop:case25} have isomorphic squares.

\begin{prop} Let $X_1, X_3\subset\mathbb{A}^3=\mathrm{Spec}(\mathbb{C}[x,y,t])$ be the (nonisomorphic by Proposition \ref{prop:case25}) hypersurfaces defined by the equation $t(x^{2}+y^{5})=1$ and $t^3(x^{2}+y^{5})=1$, respectively. Then,  $X_1\times X_1$ and $X_3\times X_3$ are isomorphic.  
\end{prop}

\begin{proof} It suffices to apply Corollary \ref{cor:squares} with  $m=10$, $a=3$ and $b=67$.
\end{proof}

Finally, let us remark that such phenomenon can not occur for affine curves.

\begin{prop}\label{prop:carre-courbes} Two smooth complex affine curves $C_1$ and $C_2$ have isomorphic squares $C_1\times C_1\simeq C_2\times C_2$ if and only if they are isomorphic. 
\end{prop}

\begin{proof}
Suppose that there exists an isomorphism $\varphi:C_{1}\times C_{1}\stackrel{\sim}{\rightarrow}C_{2}\times C_{2}$.
Then the restriction of the first projection $\mathrm{pr}_{1}:C_{2}\times C_{2}\rightarrow C_{2}$
to the image by $\varphi$ of a fiber of the first or the second projection
$\mathrm{pr}_{i}:C_{1}\times C_{1}\rightarrow C_{1}$, $i=1,2$, defines
a dominant morphism $\pi_{2}:C_{1}\rightarrow C_{2}$. Exchanging
the roles of $C_{1}$ and $C_{2}$, we obtain in a similar way a dominant
morphism $\pi_{1}:C_{2}\rightarrow C_{1}$. These extend to finite
morphisms $\overline{\pi}_{i}:\overline{C}_{j}\rightarrow\overline{C}_{i}$
between the smooth projective models of $C_{1}$ and $C_{2}$, and
we deduce from Riemann-Hurwitz formula that $\overline{C}_{1}$ and
$\overline{C}_{2}$ have the same genus $g$. Since $C_{i}$ is affine,
$\overline{C}_{i}\setminus C_{i}$ is non-empty, consisting of a finite
number of points $p_{i,j}$, $j=1,\ldots,r_{i}$, $i=1,2$. By Kunneth
formula, 
\[
H_{2}(C_{i}\times C_{i};\mathbb{Z})\simeq H_{1}(C_{i};\mathbb{Z})\otimes H_{1}(C_{i};\mathbb{Z})\simeq \mathbb{Z}^{2g+r_i-1}\otimes \mathbb{Z}^{2g+r_i-1},
\]
 where we have used that $H_{2}(C_{i};\mathbb{Z})=0$ because $C_{i}$
is affine. Thus $r_{1}=r_{2}=r\geq1$. 

If $g\geq2$ then Riemann-Hurwitz formula actually implies that $\overline{\pi}_{1}$
and $\overline{\pi}_{2}$ are isomorphisms. Thus $\pi_{2}:C_{1}\rightarrow C_{2}$
and $\pi_{1}:C_{2}\rightarrow C_{1}$are both open immersions and
are both surjective for otherwise $\pi_{1}\circ\pi_{2}$ would be
a strict open embedding of $C_{1}$ into itself, which is impossible.
So $\pi_{1}:C_{2}\rightarrow C_{1}$ and $\pi_2$ are isomorphisms by virtue of
Zariski Main Theorem. 

If $g=1$, then by Riemann-Hurwitz formula again, $\overline{\pi}_{1}:\overline{C}_{2}\rightarrow\overline{C}_{1}$
is a finite unramified morphism, say of degree $d\geq1$. Since $\overline{\pi}_{1}$
is the extension of a morphism $\pi_{1}:C_{2}\rightarrow C_{1}$,
$\overline{\pi}_{1}^{-1}(\overline{C}_{1}\setminus C_{1})\subset\overline{C}_{2}\setminus C_{2}$
and so $dr=d\cdot\sharp(\overline{C}_{1}\setminus C_{1})\leq\sharp(\overline{C}_{2}\setminus C_{2})=r$.
Thus $d=1$, and the conclusion follows from the same argument as
above.

If $g=0$, then $\overline{\pi}_{1}:\overline{C}_{2}\simeq\mathbb{P}^{1}\rightarrow\overline{C}_{1}\simeq\mathbb{P}^{1}$
is a finite morphism of degree $d\geq1$. If $d=1$, then by the same
argument again, $\pi_{1}:C_{2}\rightarrow C_{1}$ is an isomorphism.
Otherwise, if $d>1$, then since $\overline{\pi}_{1}^{-1}(\overline{C}_{1}\setminus C_{1})\subset\overline{C}_{2}\setminus C_{2}$
and $\sharp(\overline{C}_{1}\setminus C_{1})=\sharp(\overline{C}_{2}\setminus C_{2})=r\geq1$,
it must be that $\overline{\pi}_{1}$ is totally ramified over every
point of $\overline{C}_{1}\setminus C_{1}$, with $\overline{\pi}_{1}(\overline{C}_{2}\setminus C_{2})=\overline{C}_{1}\setminus C_{1}$.
By Riemann-Hurwitz formula, we have 
\[
2=2d-\sum_{p\in\overline{C}_{1}\setminus C_{1}}(d-1)-\delta=2d-r(d-1)-\delta
\]
for some $\delta\geq0$. Rewriting this equality in the form $(d-1)(2-r)=\delta\geq0$,
we conclude that either $r=1$, in which case $C_{1}\simeq C_{2}\simeq\mathbb{A}^{1}$,
or $r=2$ and then $C_{1}\simeq C_{2}\simeq\mathbb{A}_{*}^{1}$.  
\end{proof}


\end{document}